\documentclass[12pt]{amsart}
\author{Tom Enkosky}
\title{Counting points of slope varieties over finite fields}

\usepackage[margin=1in]{geometry}
\usepackage[dvips]{graphicx}
\usepackage{mathrsfs}
\usepackage[all]{xy}
\usepackage{amssymb}
\usepackage{amsmath}
\usepackage{latexsym}
\usepackage{amsfonts}
\usepackage{amsthm}
\usepackage{verbatim}
\usepackage{graphics}
\usepackage{epsf, epsfig, psfrag} 

\newcommand{\ol}{\overline}
\newcommand{\st}{~|~}

\newcommand{\Ff}{\mathbb{F}}
\newcommand{\Zz}{\mathbb{Z}}

\newtheorem{thm}{Theorem}
\newtheorem{prop}[thm]{Proposition}
\newtheorem{cor}[thm]{Corollary}

\theoremstyle{definition}

\newtheorem{defn}[thm]{Definition}
\newtheorem{example}[thm]{Example}

\newcommand{\Ra}{\Rightarrow}
\newcommand{\La}{\Leftarrow}

\newcommand{\ds}{\displaystyle}

\begin{document}
\maketitle

\begin{abstract} The slope variety of a graph is an algebraic set 
whose points correspond to drawings of a graph.  A
complement-reducible graph (or cograph) is a graph without an induced
four-vertex path.  We construct a bijection between the zeroes of the
slope variety of the complete graph on $n$ vertices over
$\mathbb{F}_2$, and the complement-reducible graphs on $n$ vertices.
\end{abstract}

\section{Introduction}
Fix a field $\Ff$ and a positive integer~$n$.  Let
$P_1=(x_1,y_1),\dots,P_n=(x_n,y_n)$ be points in the plane $\Ff^2$
such that the $x_i$ are distinct.  Let $L_{1,2},\dots,L_{n-1,n}$ be
the $\binom{n}{2}$ lines in $\Ff^2$ where $L_{i,j}$ is the line
through $P_i$ and $P_j$.  The \emph{slope variety} $S_n(\Ff)$ is the
set of possible $\binom{n}{2}$-tuples $(m_{1,2},\dots,m_{n-1,n})$,
where $m_{i,j}=\frac{y_i-y_j}{x_i-x_j}$ denotes the slope of
$L_{i,j}$.  Over an algebraically closed field of characteristic zero,
the slope variety is the set of simultaneous solutions of certain
polynomials $\tau_W$, called \emph{tree polynomials}
\cite{GGV,Slopes}, indexed by wheel subgraphs of the complete graph
$K_n$.  (A $k$-wheel is a graph formed from a cycle of length~$k$ by
introducing a new vertex adjacent to all vertices in the cycle.)  It
is conjectured, and has been verified experimentally for $n\leq 9$,
that the ideal of all tree polynomials is in fact generated by the
subset $\{\tau_Q\}$ \cite{Slopes}, where $Q$ is a 3-wheel
(equivalently, a 4-clique) in $K_n$.

The tree polynomials have integer coefficients, which raises the
question of counting their solutions over a finite field.  Let $\Ff_q$
be the field with $q$ elements.  In this article, we count the
solutions of the tree polynomials over $\Ff_2$ and give some
generalizations for $q>2$.

\begin{thm}\label{MainTheorem}
Let $n$ be a positive integer and let
$\Ff_2[K_n]:=\Ff_2[m_{1,2},\dots,m_{n-1,n}]$.  Let $I_n$ denote the
ideal of $\Ff_2[K_n]$ generated by the tree polynomials of wheel
subgraphs of $K_n$, and let $J_n$ denote the ideal generated by the
tree polynomials of $K_4$-subgraphs of $K_n$ (so $J_n\subseteq I_n$).

Then the following sets are equinumerous:
\begin{enumerate}
\item \label{ZIn} the zeroes of $I_n$, i.e., the points in
  $\Ff_2^{\binom{n}{2}}$ on which all tree polynomials vanish;
\item \label{ZJn} the zeroes of $J_n$, i.e., the points in 
  $\Ff_2^{\binom{n}{2}}$ on which all tree polynomials of 3-wheels vanish;
\item \label{cog} complement-reducible graphs (or ``cographs'') on vertex 
  set~$[n]=\{1,2,\dots,n\}$, that is, graphs on $[n]$ having no induced 
  subgraph isomorphic to a four-vertex path;
\item \label{swq} switching-equivalence classes of graphs on vertex
  set $[n+1]$ such that no member of the class contains an induced
  5-cycle.
\end{enumerate}
\end{thm}

We will explain all these combinatorial interpretations below.  The
following Theorem appears in \cite[Exercise 5.40]{EC2} and is credited
to Cameron \cite{Cameron}.

\begin{thm}\label{SPN-thm}
  The following sets are equinumerous:
  \begin{enumerate}
  \item \label{swq2} switching-equivalence classes of graphs on vertex
    set $[n+1]$ such that no member of the class contains an induced
    5-cycle;
  \item \label{spp} series-parallel posets with $n$ labeled vertices;
  \item \label{spn} series-parallel networks with $n$ labeled edges.
  \end{enumerate}
\end{thm}

In this paper, we use the special structure of tree polynomials to
prove first the equality of (\ref{ZIn}), (\ref{ZJn}) and (\ref{cog})
of Theorem \ref{MainTheorem} (Proposition \ref{ZeroCograph}), and then
a bijection between (\ref{cog}) and (\ref{swq}) of Theorem
\ref{MainTheorem} (Proposition \ref{Cog5Cyc}).

We note that a bijection between \emph{unlabeled} complement-reducible
graphs and \emph{unlabeled} series-parallel networks was given by
Sloane, see sequence A000084 \cite{OEIS}.  We have not found in the
literature an explicit bijection for the corresponding labeled
objects.  The numbers of points in $S_2(K_1)$, $S_2(K_2)$, \dots, are
  $$1,\ 2,\ 8,\ 52,\ 472,\ 5504,\ 78416,\ \dots$$
which is sequence A006351 in \cite{OEIS}.

\section{Background}
\subsection{Graph theory}
We list some necessary notation here; for a general background on
graph theory see \cite{Bollobas} or \cite{West}.  A \emph{graph} $G$
is an ordered pair $(V,E)$ of \emph{vertices} and \emph{edges} i.e.,
$V$ is a finite set, and $E$ is a set of 2-subsets of $V$.  Two
vertices $u,v\in V$ are adjacent if there is an edge $uv\in E$ between
them.  If $V,E$ are not specified then $V(G)$ is the set of vertices
of $G$ and $E(G)$ is the edge set.  All graphs in this paper are
\emph{simple}, i.e., they have no loops or multiple edges.

For $U\subseteq V$, the \emph{induced subgraph} $G|_U$ of $G$ on $U$,
is the graph with vertex set $U$ and edge set $\{uv\in E(G)\st u,v\in
U\}$.  The \emph{intersection} $G\cap H$ of two graphs $G$ and $H$ is the
graph with vertex set $V(G)\cap V(H)$ and edge set $E(G)\cap E(H)$.
The \emph{complement} $\ol{G}$ of $G$ is the graph on the same set of
vertices as $G$ whose edges are exactly the non-edges of $G$.

Let $K_n$ denote the complete graph on $n$ vertices.  Let $P_n$ denote
the path on $n$ vertices, also called the $n$-path.  A
\emph{complement-reducible graph}, or \emph{cograph}, has no induced
$P_4$.  An important fact that we will need is that $G$ is
complement-reducible if and only if for every induced subgraph
$H\subseteq G$, either $H$ or the complement $\ol{H}$ is disconnected;
see \cite{Corneil}.

The \emph{$k$-wheel} $W(v_0;v_1,\dots,v_k)$ is the graph with vertices
$\{v_0,\dots,v_k\}$ and edges $v_0v_1,\dots,v_0v_k$,
$v_1v_2,\dots,v_{k-1}v_k,v_kv_1$, $k\geq 3$.  Note that the wheel
$W(v_0;v_1,\dots,v_k)$ is invariant up to dihedral permutations of
$v_1,\dots,v_k$.  The vertex $v_0$ is called the \emph{center}; the
other vertices are called the \emph{spokes}.  The edges incident to
the center are called the \emph{radii}, and the other edges are
\emph{chords}.  Note that a 3-wheel is the complete graph on four
vertices.

\subsection{Series-parallel networks}
A \emph{network} is a graph $G$ with two vertices $s_G,t_G$ designated
as the \emph{source} and \emph{sink}, respectively.  Two networks $G$
and $H$ can be connected in \emph{series} or \emph{parallel}.  The
\emph{series connection} $G\oplus H$ is defined by identifying $t_G$
with $s_H$, and designating $s_G$ as the source and $t_H$ as the sink.
The \emph{parallel connection} $G+H$ is defined by identifying $s_G$
with $s_H$ and $t_G$ with $t_H$.

A \emph{series-parallel network} is a graph obtained from the
following rules:
\begin{enumerate}
\item a graph with one edge $st$ is a series-parallel network;
\item if $G$ and $H$ are series-parallel networks, then $G\oplus H$ and
$G+H$ are series-parallel networks.
\end{enumerate}

One can define series and parallel connections for posets in a
similar fashion; see \cite[Section~3.2]{EC1}.  Two posets $P$ and $Q$
are connected in series by taking their \emph{ordinal sum} $P\oplus
Q$: declaring that all elements of $Q$ are larger than all elements of
$P$ (or vice versa) leaving all other relations unchanged.  The two
posets are connected in parallel by taking the disjoint union.  A
\emph{series-parallel poset} is a poset built up from single-element
posets by series and parallel extensions.

Let $s(n)$ be the number of labeled series-parallel networks on $n$
vertices.  The sequence begins
  $$s(1)=1,\quad s(2)=2,\quad s(3)=8,\quad s(4)=52,\quad
s(5)=472,\quad s(6)=5504,\quad\dots$$ This is sequence A006351 in the
On-Line Encyclopedia of Integer Sequences \cite{OEIS}.

\subsection{Switching equivalence}\label{SwitchSec}
Let $G$ be a graph on $[n+1]$ and let
$X\subseteq [n]$.  The \emph{switch} of $G$ with respect to $X$ is the
graph $s_X(G)$ on $[n+1]$ whose edges $e$ satisfy
one of two conditions:
\begin{enumerate}
\item $e\in E(G)$ and either both vertices of $e$ belong to $X$ or
neither do;
\item $e\not\in E(G)$ and exactly one vertex of $e$ belongs to $X$.
\end{enumerate}
This operation is also referred to as graph switching or Seidel
switching \cite{GlosSgnGain}.  Let $\mathscr{G}_{n+1}$ be the set of
graphs on $[n+1]$.  Then switching defines an action of $\ds \Zz_2^n$
on $\mathscr{G}_{n+1}$.  For $\ds x=(x_1,\dots,x_n)\in\Zz_2^n$ let
$X=\{i\st x_i=1\}\subset [n]$.  Then the group action is $xG=s_X(G)$.
This action is free because $s_X(G)=G$ if and only if $X=\emptyset$.
The orbits are called \emph{switching classes}, denoted by $[G]$.  To
see that each orbit contains exactly one graph in which the vertex
$n+1$ is isolated, let $G\in\mathscr{G}_{n+1}$ and let $X=N(n+1)$ be
the set of neighbors of $n+1$.  Then the graph $s_X(G)$ has $n+1$ as
an isolated vertex.  On the other hand if $X$ is any other subset of
$[n]$, then $n+1$ will be adjacent to some vertex of $s_X(G)$.
The number of switching classes on $[n+1]$ is $s(n)$, the number of
labeled series-parallel networks \cite[Exercise~5.40(b)]{EC2},
\cite{Cameron}.

\subsection{Tree polynomials}
We briefly sketch the basics of graph picture spaces; for more details, see
\cite{GGV}.  
\begin{defn}
Let $G=(V,E)$ be a graph.  For each $e\in E$, let $m_e$ be a variable. 
For each subset $F\subseteq E$ define
$$m_F=\prod_{f\in F}m_f.$$
We regard the square-free monomial $m_F$ as corresponding to the spanning 
subgraph $(V,F)$, and we will often ignore the distinction between the monomial
and the graph.
\end{defn}
A \emph{picture} of a graph $G=(V,E)$ is a collection of
labeled points and lines in the plane, corresponding to the vertices and 
edges of $G$, respectively, such that the line $\ell_e$ corresponding to 
an edge $e$ contains both points corresponding to the endpoints of~$e$.  
Provided that no lines are vertical, each line $\ell_e$ has a well-defined 
slope $m_e$, and so each picture determines a \emph{slope vector} 
$(m_e)_{e\in E}$.  The set of all possible slope vectors is called the 
\emph{slope variety} $S(G)$.  

The slope variety is the set of common zeroes of the set of polynomials called
\emph{tree polynomials}, as we now explain.
A \emph{(rigidity) pseudocircuit} is a graph $H$ whose edge set can be
partitioned into two spanning trees.  A \emph{coupled spanning tree} of $H$
is a tree whose complement is also a spanning tree; the set of all
coupled  spanning trees of $H$ is denoted $Cpl(H)$.
For each pseudocircuit $H\subseteq G$, there is a polynomial
\begin{align}\label{cpltreepoly}
\tau_H=\sum_{T\in Cpl(H)}\epsilon(H,T)m_T
\end{align}
that vanishes on the slope variety of $G$;  where each
$\epsilon(H,T)\in\{1,-1\}$.  Call this polynomial the  \emph{tree polynomial}
of $H$.

Because the tree polynomials have integer coefficients, it makes sense 
to consider these polynomials inside the polynomial ring
$$\Ff_q[G]:=\Ff_q[m_e\st e\in E(G)].$$
Define the \emph{q-slope variety} $S_q(G)$ to be the zero set of 
the ideal generated by the tree 
polynomials of all pseudocircuit subgraphs of $G$.  The main concern
of this paper is $S_2(K_n)$, the set of zeroes of the complete graph 
over $\Ff_2$. 

The most important pseudocircuits are the wheels.  
The tree polynomial of the wheel $W=W(v_0;v_1,\dots,v_k)$ has the form
\begin{equation}
\label{treepoly}
\tau_W=\underbrace{\prod_{i=1}^k(m_{0,i}-m_{i,i+1})}_{\tau_1}
-\underbrace{\prod_{i=1}^k(m_{0,i}-m_{i-1,i})}_{\tau_2}
\end{equation}
where $m_{k,k+1}=m_{1,k}$ \cite[eqn.~(6)]{Slopes}.

Suppose we draw the wheel $W(v_0;v_1,\dots,v_k)$ with $v_0$ in the
center and the indices of the spokes increasing as we travel clockwise
around the perimeter.  Each binomial factor in $\tau_1$ is a radius
minus the adjacent chord pointing in the clockwise direction, whereas
each binomial factor in $\tau_2$ is a radius minus the adjacent chord
pointing in the counter-clockwise direction.  Therefore, if we expand
the expression \eqref{treepoly} for $\tau_W$, then the star subgraph
and the cycle of all the chords each occur twice, and with opposite
signs.  The only remaining terms are coupled spanning trees, which are
obtained by picking a nontrivial subset of radii along with all chords
pointing clockwise or counterclockwise, but not both.  See Figure
\ref{Comp-Spn-tree}.
\begin{figure}[hbt]
      \resizebox{2.1in}{1.4in}{\includegraphics{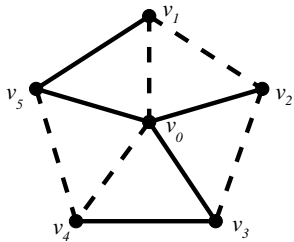}}
      \caption{Two complementary coupled spanning trees of a 5-wheel}
      \label{Comp-Spn-tree}
\end{figure}

The tree polynomials of all the wheels in $K_n$ generate the ideal of
tree polynomials of all rigidity pseudocircuits in $K_n$ \cite{Slopes}.
Define ideals $I_n,J_n\subseteq \Ff_2[G]$ as follows:
\begin{align*}
I_n &= (\tau_W \st W\text{ is a wheel in } K_n),\\
J_n &= (\tau_Q \st Q\subseteq K_n \text{ is isomorphic to } K_4).
\end{align*}
It was conjectured in \cite{Slopes} that $I_n=J_n$ when considered as
ideals over $\mathbb{C}$.  Using the computer algebra system Macaulay
this conjecture has been verified for $n\leq 9$ \cite{Slopes}.

\section{A bijection between slope vectors and complement-reducible graphs}
In this section we count the points of $S_2(K_n)$, the slope variety
of $K_n$ over $\mathbb{F}_2$.  The points of $\Ff_2^{\binom{n}{2}}$
have their coordinates indexed by the edges of $K_n$ and have value
either 0 or 1, which motivates the following notation:

\begin{defn}
Let 
$a=(a_{1,2},a_{1,3},\dots,a_{n-1,n})\in\Ff_2^{\binom{n}{2}}$.
We define the graph $G_a$ to be the graph on $[n]$ with edge set
$\ds E(G_a)=\{ij \st a_{i,j}=1\}$. 
\end{defn}

\begin{prop}\label{TreeNotZero}
Let $W=W(v_0;v_1,\dots,v_k)$ be a wheel and 
$a\in\Ff_2^{\binom{n}{2}}$.  Then $\tau_W(a)\neq 0$
if and only if $H_a:=G_a\cap W$ is a coupled spanning tree of $W$.
\end{prop}

\begin{proof}
$(\La)$ Suppose that $H_a$ is a coupled spanning tree of $W$.  When
$\tau_W$ is written in the form of equation (\ref{cpltreepoly}), over
$\Ff_2$, it is the sum of all coupled spanning trees of $W$.
Evaluating $\tau_W$ at $a$ gives exactly one non-zero term, 
hence $\tau_W(a)\neq 0$.

$(\Ra)$ Suppose $\tau_W(a)\neq 0$. First we show that $H_a$ is a
spanning tree of $W$.  Over $\Ff_2$ exactly one of $\tau_1(a)$ or
$\tau_2(a)$ has value 1, say $\tau_1(a)=1$.  Each binomial factor of
$\tau_1(a)$ must contain exactly one variable with value 1.
Therefore, $H_a$ contains exactly $k$ edges, which is the number of
edges of a spanning tree of $W$.  In order to show that $H_a$ is a
spanning tree it is enough to show that it is acyclic.  If $H_a$
contains a cycle $C$, then either there exist $i$ and $j$, $1\leq
i<j\leq k$, such that $v_0v_i,v_iv_{i+1},v_{j-1}v_j,v_0v_j\in E(C)$,
or $C$ is the set of chords of $W$.  In the first case, both terms
$\tau_1(a)$ and $\tau_2(a)$ have value 0 because $m_{0i}-m_{i(i+1)}$
is a factor of $\tau_1$ and $m_{0j}-m_{(j-1)j}$ is a factor of
$\tau_2$.  In the second case, formula (\ref{treepoly}) will be
$$\tau_W(a)=\prod_{i=1}^k(a_{0i}-1)-\prod_{i=1}^k(a_{0i}-1)=0.$$

Now we show that $H_a$ is in fact a \emph{coupled} spanning tree of
$W$.  Define $\ol{a}\in\Ff_2^{\binom{n}{2}}$ by $\ol{a_{ij}}=1-a_{ij}$
for all $1\leq i< j\leq n$.  Therefore, $G_{\ol{a}}=\ol{G_a}$ is the
complement of $G_a$.  If $\tau_W(a)\neq 0$ then $\tau_W(\ol{a})\neq 0$
because each binomial factor of $\tau_i(\ol{a})$ will have the same
value as in $\tau_i(a)$, for $i=1,2$.  Therefore
$H_{\ol{a}}=\ol{H_a}\cap W$ is a spanning tree of $W$, hence $H_a$ is
a coupled spanning tree of $W$.
\end{proof}

\begin{prop}\label{ZeroCograph}
Let $\ds a\in\Ff_2^{\binom{n}{2}}$.  The
following are equivalent:
\begin{enumerate}
\item\label{In} $a$ is a zero of $I_n$;
\item\label{Jn} $a$ is a zero of $J_n$;
\item\label{cograph} $G_a$ is a complement-reducible graph.
\end{enumerate}
\end{prop}

\begin{proof}
$(\ref{In}\Ra\ref{Jn})$ This implication follows from 
the containment $J_n\subseteq I_n$.

$(\ref{Jn}\Ra\ref{cograph})$ Suppose $\ds a\in\Ff_2^{\binom{n}{2}}$ is
a zero of $J_n$.  By Proposition \ref{TreeNotZero}, if $W\subseteq
K_n$ is any 3-wheel (and hence isomorphic to $K_4$), then $G_a\cap W$
is not a coupled spanning tree of $W$.  Since every coupled spanning
tree of $K_4$ is isomorphic to $P_4$ (the only spanning trees of $K_4$
are isomorphic to $P_4$ or the three edge star), $G_a$ does not
contain an induced $P_4$.

$(\ref{cograph}\Ra\ref{In})$ Let $\ds a\in\Ff_2^{\binom{n}{2}}$ be such 
that $G_a$ is a complement-reducible graph.  Let $W\subseteq K_n$ be a wheel with $V=V(W)$.
Either $G_a|_V$ or $\ol{G_a}|_V$ is disconnected, because $G_a$ is a 
complement-reducible graph.  Therefore either $G_a\cap W$ or 
$G_{\ol{a}}\cap W$ is disconnected. Since these two graphs are complementary
subgraphs of $W$, neither one is a coupled spanning tree.  Therefore, by 
Proposition \ref{TreeNotZero}, $\ds \tau_W(a)=0$ for every wheel 
$W\subseteq K_n$.
\end{proof}

\section{A bijection between complement-reducible graphs and switching classes}
In this section, we establish a bijection (Proposition \ref{Cog5Cyc}) between
the set of
graphs on $n$ labeled vertices with an induced $P_4$, and the switching
classes on $n+1$ labeled vertices containing a graph with an induced 5-cycle.  
Recall from Section \ref{SwitchSec} that each switching class contains 
exactly one graph in which the vertex $n+1$ is isolated.  Therefore the 
bijection from the set of graphs on $[n]$ to the switching classes 
on $[n+1]$ is given by sending $G\subseteq K_n$ to $[G]$, the orbit 
containing $G$.

\begin{prop}\label{Cog5Cyc}
Let the additive group $\Zz_2^n$ act on $\mathscr{G}_{n+1}$ by switching 
as described in Section \ref{SwitchSec}.  Then:
\begin{enumerate}
\item\label{Injective} If $G\in\mathscr{G}_{n+1}$ has an induced 5-cycle, 
then every $H\in [G]$ has an induced 4-path.
\item\label{Surjective} If $G\in\mathscr{G}_n$ has an induced
4-path, then, regarding $G$ as a graph on $[n+1]$ by introducing $n+1$ as
an isolated vertex, there is an $H\in \mathscr{G}_{n+1}$ such that
$G\in [H]$ and $H$ has an induced 5-cycle.
\end{enumerate}
\end{prop}

\begin{proof}
(\ref{Injective}) Let $G\in\mathscr{G}_{n+1}$ have an induced 5-cycle
  $C=\{v_1,\dots,v_5\}$, and let $X\subseteq [n]$.  If $|V(C)\cap X|<
  2$, then four of the vertices, say $U=\{v_1,v_2,v_3,v_4\}$, are in
  $[n]\setminus X$.  Switching by $X$ does not affect the induced
  subgraph on $U$.  Similarly, if $|V(C)\cap X|>3$, then
  $(s_X(G)\st_U)\cong P_4$.

Suppose $|V(C)\cap X|=2$.  Without loss of
generality we may assume either $X=\{v_2,v_5\}$ or $X=\{v_3,v_4\}$. 
In both cases $v_5v_3v_4v_2$ is an induced 4-path in $s_X(G)$, as
shown in the figure.
\begin{figure}
      \resizebox{4in}{1.5in}{\includegraphics{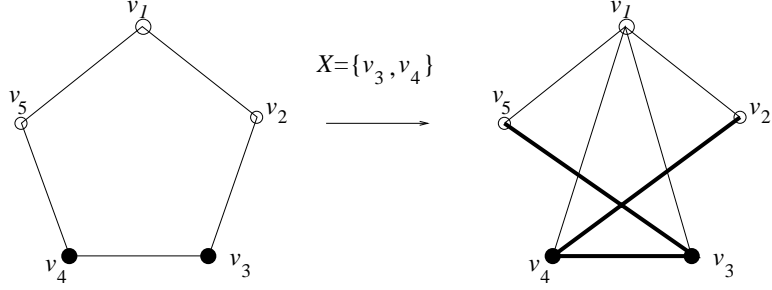}}
      \caption{The action with $X=\{v_3,v_4\}$}
      \label{first-switch}
\end{figure}
\begin{figure}
      \resizebox{4in}{1.5in}{\includegraphics{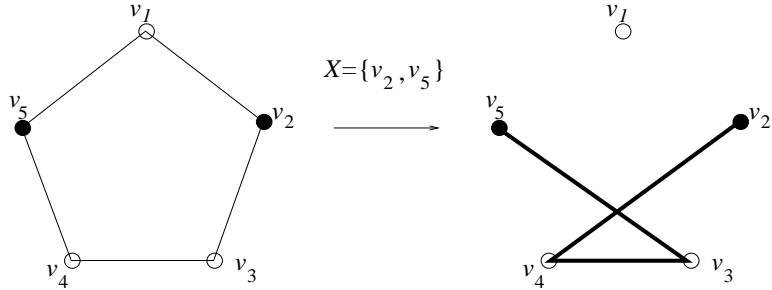}}
      \caption{The action with $X=\{v_2,v_5\}$}
      \label{second-switch}
\end{figure}
If $|V(C)\cap X|=3$, then $|V(C)\cap ([n+1]\setminus X)|=2$.  The same
results as above will hold for this case, therefore $s_X(G)$ has an induced
$P_4$.

(\ref{Surjective}) Suppose $v_5v_3v_4v_2$ is an induced $P_4$ in
$G\subseteq K_n$ and $X=\{v_2,v_5\}$.  Then $s_X(G)$
has the induced 5-cycle $C=\{v_1,\dots,v_5\}$ with $v_1=n+1$.
\end{proof}

\section{Counting points over other finite fields}
It is natural to ask whether these techniques can be extended to enumerate
points of the slope variety $S_n(\Ff_q)$ over  $\Ff_q$.  This problem 
appears to be difficult, because the zeroes of a tree
polynomial over an arbitrary field do not seem to admit a uniform 
graph-theoretic description as they do over $\Ff_2$.  In this section, 
we describe 
some partial progress in this direction, and explicitly work out the simplest
nontrivial case ($n=4$, $q=3$) to illustrate the kinds of difficulties 
involved.

A point in $\Ff_q^{\binom{n}{2}}$ 
corresponds to an \emph{$\Ff_q$-weighted $K_n$}, that is, a copy of
$K_n$ whose edges are assigned weights in $\Ff_q$.
For $a\in \Ff_q^{\binom{n}{2}}$ define $G_a$ to be the $\Ff_q$-weighted
$K_n$ where edge $ij$ is given weight $a_{ij}$.
We say that $G_a$ has a \emph{weight-induced subgraph} 
$H$ if there is some value $\alpha\in\Ff_q$ such that 
$$E(H)=\{e\in E(K_n)\st a_e=\alpha\}.$$

One possible  approach to generalizing the previous results would be 
to define a $q$-analogue to switching.  
Let the additive group $\Ff_q^n$ act on $\Ff_q^{\binom{n}{2}}$ by
$$((x_1,\dots,x_n)\cdot a))_{ij}=(a_{ij}+x_i+x_j).$$ If $q=2$, then
this is exactly the switching action described in Section
\ref{SwitchSec}.  Note that this is not the same definition of
$q$-switching given by Zaslavsky \cite{GlosSgnGain}.  The present
definition seems more likely to be relevant in the context of slopes
because it does not rely on an orientation of the edges.  (Recall that
the weight of an edge is the slope of the corresponding line segment
in a picture of $K_n$; the slope does not depend on the direction in
which way the edge is traversed.)  One would hope to generalize the
$q=2$ case by describing the points of $S_n(\Ff_q)$ in terms of
forbidden weight-induced subgraphs.  Though this does not appear to
work in general, some facts do carry over to the setting of an
arbitrary finite field.

\begin{prop}\label{Generalize}
Let $W=W(v_0;v_1,v_2,v_3)$ be a 3-wheel, and let $a\in\Ff_q^{\binom{4}{2}}$
be a point whose coordinates correspond to assigning weights to the 
edges of $W$.  Then:
\begin{enumerate}
\item\label{GenNot0} If $G_a$ has a weight-induced $P_4$, then $a$ is
  not a zero of $\tau_W$.
\item\label{GenIs0} If $G_a$ has a weight-induced claw (that is, a
  star with three edges), then $a$ is a zero of $\tau_W$.
\item\label{cycIs0} If $G_a$ has a weight-induced cycle, then $a$ is a
  zero of $\tau_W$.
\end{enumerate}
\end{prop}

\begin{proof}
(\ref{GenNot0}) Suppose that $G_a$ has a weight-induced $P_4$.   
The induced subgraph on the vertices of that $P_4$ can be drawn as
in Figure \ref{weight-p4}.
\begin{figure}
  \resizebox{2.1in}{1.4in}{\includegraphics{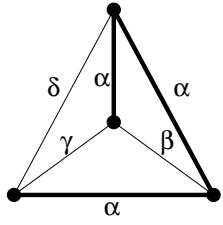}}
  \caption{A weight induced $P_4$}
  \label{weight-p4}
\end{figure}
where $\alpha,\beta,\gamma,\delta\in\Ff_q$, and $\alpha$ does not
equal any of the other values.
Then,
$$\tau_W(a)=(\alpha-\alpha)(\beta-\alpha)(\gamma-\delta)-
(\alpha-\delta)(\beta-\alpha)(\gamma-\alpha)\neq 0.$$

(\ref{GenIs0}) Suppose that $G_a$ has a weight-induced star $S_4$,
whose edges have the weight $\alpha\in\Ff_q$.  If we draw $W$ so that
the center of the star is the center of the wheel, then
$$\tau_Q(a)=(\alpha-\beta)(\alpha-\gamma)(\alpha-\delta)-
(\alpha-\gamma)(\alpha-\delta)(\alpha-\beta)=0,$$
for some $\beta,\gamma,\delta\in\Ff_q$.

(\ref{cycIs0}) Suppose that $G_a$ has a weight-induced cycle $C$.  The
graph $W$ can be drawn so that $C$ contains the vertex $v_0$.  Then,
for some $1\leq i<j\leq 3$, the edges $v_0v_i$, $v_iv_{i+1}$, $v_0v_j$,
$v_jv_{j-1}$ all have the same weight $\alpha$. (Note that if the
cycle is a 3-cycle then $v_{j-1}=v_i$.)  Then both $\tau_1(a)$ and
$\tau_2(a)$ contain the factor $\alpha-\alpha$, so $\tau_W(a)=0$.
\end{proof}

\begin{cor}
Let $a\in\Ff_q^{\binom{n}{2}}$.  If $G_a$ contains a weight-induced
$P_4$, then $a$ is not a zero of $I_n$ over $\Ff_q$.  Conversely,
if every 4-clique of $G_a$ contains a weight-induced cycle or a 
weight-induced star $S_4$, then $a$ is a zero of $I_n$ over $\Ff_q$.
\end{cor}

\begin{example}
Let $W=W(v_0;v_1,v_2,v_3)$ be a 3-wheel.  We use Proposition 
\ref{Generalize} to count the number of zeroes of $\tau_W$ over $\Ff_3$.

If some value occurs at least four times in $a$, then $\tau_W(a)=0$
because $G_a$ has a weight-induced cycle.  If some value $\alpha$
occurs exactly three times in $a$, then $\tau_W(a)\neq 0$ if and only
if the weight-induced graph on $\alpha$ is a 4-path.  The cases where
each value of $a$ occurs two times are not covered by Proposition
\ref{Generalize}, so we must consider them separately.  For distinct
$\alpha,\beta,\gamma\in\Ff_3$ there are three possibilities, up to a
relabeling of the vertices; see Figure \ref{type-222}.
\begin{figure}[hbt]
  \resizebox{3.8in}{1.1in}{\includegraphics{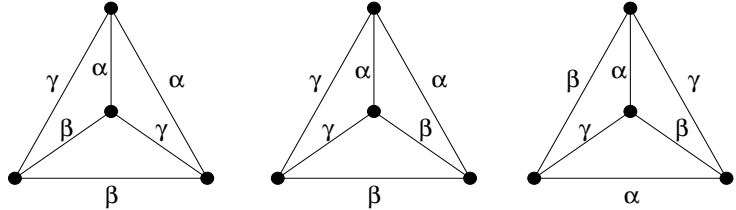}}
  \caption{The first weight corresponds to a zero of $\tau_{W}$, but
    the second two do not.}
  \label{type-222}
\end{figure}

Define the \emph{type} of $a\in\Ff_q^{d}$ to be the partition whose
  parts are the numbers of occurrences of each element of $\Ff_q$
  among the entries of $a$.  Some simple counting gives the following
  table:
\begin{center}
\begin{tabular}{c|c|c}
Type & Number of zeroes & Number of non-zeroes  \\\hline
$(6)$ & $3$ & $0$ \\
$(5,1)$ & $36$ & $0$\\
$(4,2)$ &  $90$ & $0$ \\
$(4,1,1)$ & $90$ & $0$ \\
$(3,3)$ & $24$ & $36$\\
$(3,2,1)$ & $144$ & $216$\\
$(2,2,2)$ & $36$ & $54$ \\\hline
Total & $423$ & $306$
\end{tabular}
\end{center}
\bigskip
\end{example}

If $q>3$, then there are more cases to check which are not covered by
Proposition \ref{Generalize}.  Using the computer algebra software
Maple, one can check that over $\Ff_3$ the number of zeroes of $I_4$
and $I_5$ are 423 and 9243, respectively.  Over $\Ff_5$ the numbers
are 4909, 262645, respectively.  It is not clear what combinatorial
structure (analogous to complement-reducible graphs) might count these
points; for instance, these numbers do not appear in the Encyclopedia
of Integer Sequences \cite{OEIS}.

\bibliographystyle{amsplain}
\bibliography{biblio}

\end{document}